\documentclass[smallextended]{svjour3}

\usepackage{amsmath}
\usepackage{amssymb}
\usepackage[mathscr]{eucal}
\usepackage{graphicx}
\usepackage{xcolor}
\usepackage{hyperref}
\usepackage{multirow}
\usepackage{float}
\usepackage{algorithmic}

\newtheorem{algorithm}{Algorithm}[section]   
\usepackage{chngcntr}
\counterwithin*{table}{section}
\counterwithin*{figure}{section}
\counterwithin*{equation}{section}
\counterwithin*{example}{section}
\counterwithin*{theorem}{section}
\counterwithin*{lemma}{section}
\counterwithin*{corollary}{section}
\counterwithin*{remark}{section}

\smartqed

\begin{document}

\title{A new approach to shooting methods
	for terminal value problems
	of fractional differential equations}
\titlerunning{Efficient shooting methods for fractional TVPs}
\author{Kai Diethelm
	 \and Frank Uhlig}

\institute{K. Diethelm \at
	Faculty of Applied Natural Sciences and Humanities (FANG), 
 	Technical University of Applied Sciences W\"urzburg-Schweinfurt,
	Ignaz-Sch\"on-Str.\ 11,
	97421 Schweinfurt,
	Germany,
	ORCID: \href{https://www.orcid.org/0000-0002-7276-454X}{0000-0002-7276-454X}.
	\email{kai.diethelm@thws.de}
	\and
	F. Uhlig \at
	Department of Mathematics and Statistics, 
	Auburn University, 
	Auburn, AL 36849-5310, 
	USA,
	ORCID: \href{https://www.orcid.org/0000-0002-7495-5753}{0000-0002-7495-5753}.
	\email{uhligfd@auburn.edu}
}

\date{Received: date / Accepted: date}

\maketitle

\begin{abstract}
	For terminal value problems of fractional differential equations of 
	order $\alpha \in (0,1)$ that use Caputo derivatives, shooting methods
	are a well developed and investigated approach. 
	Based on recently established analytic properties of such problems,
	we develop a new technique to select the required initial values
	that solves such shooting problems quickly and accurately.
	Numerical experiments indicate that this new 
	proportional secting technique converges very quickly and accurately to the solution.
	Run time measurements indicate a speedup factor 
	of between 4 and 10 when compared to the standard bisection method.



\end{abstract}



\section{Introduction and Motivation} 
	\label{sec:intro}

Differential equations of fractional (i.e., non-integer) order \cite{Di2010} 
are an object of great current interest.  They are 
 useful tools for modeling various phenomena in science and engineering,
see, e.g., \cite{hbfc7,hbfc8,hbfc6,hbfc4,hbfc5}.  
These  problems typically have the form
\begin{equation}
	\label{eq:ivp}
	D_a^\alpha y(t) = f(t, y(t)), \qquad y(a) = \tilde y_0,
\end{equation}
where $\alpha \in (0, 1)$ is the order of the differential operator. Note that  $\alpha > 1$
arises only rarely in applications and this case will not be discussed here.
In eq.~\eqref{eq:ivp}, $f:[a,b] \times \mathbb R \to \mathbb R$ 
is a given function
and  $\tilde y_0 \in \mathbb R$  describes the initial state of the modeled system at $t = a$.
The differential operator $D_a^\alpha$ in eq.~\eqref{eq:ivp}
is the  Caputo differential operator of order $\alpha$ with starting point $a \in \mathbb R$ as
defined by \cite[Definition 3.2]{Di2010}, namely
\begin{equation}
	\label{eq:def-caputo}
	D_a^\alpha y(t) 
	= \frac 1 {\Gamma(1-\alpha)} \frac{\mathrm d}{\mathrm d t} 
		\int_a^t (t-s)^{-\alpha} \left( y(s) - y(a) \right) \mathrm d s
\end{equation}
for $t \in [a,b]$ and sufficiently smooth functions $y : [a,b] \to \mathbb R$.

In equations \eqref{eq:ivp} and \eqref{eq:def-caputo},
the starting  point $a$ plays a special role as the
start of the process that is being studied. A typical application  
from mechanics is an object made from viscoelastic material 
under external loads that is still  in its virgin state
for  $t < a$  and to which forces are only applied for
$t \ge a$. Here problem \eqref{eq:ivp} is an \emph{initial value problem} 
as $y(a) = \tilde y_0$ is the state of the process at the initial time $t = a$ and we are  
interested in finding  $y(t)$ for $t \in [a,b]$ and  some 
given time $b > a$.

From the analytical viewpoint  such initial value problems are well understood, 
see e.g., \cite[Chapters 6 and 7]{Di2010}. Many numerical methods  
have been proposed and investigated; cf., e.g., \cite{hbfc3}. However, 
from the modeling perspective, these methods often are of limited use because
they hinge on the exact state of the process
at the initial time $t=a$ and this may be impossible to determine in actual applications.
If  one can only measure the value of $y(b)$ for some $b > a$
but not at $a$ itself, this leads to
\begin{equation}
	\label{eq:tvp}
	D_a^\alpha y(t) = f(t, y(t)), \qquad y(b) = y^*.
\end{equation}
Then we are tasked to solve (\ref{eq:tvp}) on the interval $[a, c]$ where
$a$ is the starting time of the process and $c \ge b$. This can be done in two steps:   
\begin{itemize}
\item Solve the problem on the interval $[a,b]$. Since  $b$ 
	in eq.~\eqref{eq:tvp} denotes the interval's  end point,   this is a 
	\emph{terminal value problem}.
\item With  the solution known on the entire interval $[a,b]$, the value of $y$ 
	at the initial point $a$ is  known as $\tilde y_0 = y(a)$.
	Therefore we can  replace the terminal condition $y(b) = y^*$ in eq.~\eqref{eq:tvp}
	by the initial condition $y(a) = \tilde y_0$. This converts the original problem into 
	the classical initial value problem \eqref{eq:ivp} that can now be solved on the entire interval $[a,c]$.
\end{itemize}
As indicated above,
the second step of this process has a well understood structure and can be handled by 
standard methods. Therefore it does not require any special attention.  Hence we focus only
on the first step in this work.

\section{Analytic Properties of Terminal Value Problems}
\label{sec:analysis}

Here we  provide the basis for our
numerical work with terminal value problems
for fractional differential equations and recall some of their known analytical properties.

Conditions for well posed-ness have been  discussed and partially 
established in \cite{Di2008,DF2012}. A complete analysis
is  in \cite{CT2017}, with additional aspects  presented in \cite{DF2018,FM2011}.
For our work here, we shall specifically use the following result that is an immediate consequence 
of a statement by Cong and Tuan regarding initial value problems \cite[Theorem 3.5]{CT2017}.
It follows the classical set-up in assuming
that the function $f$ in eq.~\eqref{eq:tvp} is continuous on $[a,b] \times \mathbb R$, maps 
into $\mathbb R$ and
satisfies a Lipschitz condition with respect to the second variable.

\begin{theorem}
	\label{thm:tvp-sol-ex}
	Let $f : [a,b] \times \mathbb R \to \mathbb R$ be continuous and satisfy the Lipschitz condition 
	with respect to
	the second variable 
	\begin{equation}
		\label{eq:lip}
		|f(t, x_1) - f(t, x_2)| \le L(t) | x_1 - x_2 |
	\end{equation}
	for all $t \in [a,b]$ with some function $L \in C[a,b]$. 
	Then, for any terminal value $y^* \in \mathbb R$ the terminal value problem \eqref{eq:tvp}
	has a unique solution $y$ in $C[a,b]$.
\end{theorem}

\begin{remark}
	When talking about initial value problems, it is common practice to discuss not only scalar but
	multidimensional problems, i.e.\ to assume that the function $f$ on the right-hand side of the
	differential equation maps from $[a,b] \times \mathbb R^d$ to $\mathbb R^d$ with some $d \ge 1$.
	In the initial value problem setting, this generalization does not introduce any difficulties
	as far as the existence and the uniqueness of solutions are concerned. However, if terminal
	value problems are addressed as we are doing here, the situation is significantly different.
	To be precise, in \cite{CT2017} it was shown that well-posedness of terminal value problems
	in general only occurs for $d = 1$; with a counterexample for $d > 1$  given in \cite[Section 6]{CT2017}.
	This counterexample demonstrated that multiple solutions can arise when $d > 1$.
	Therefore we only consider the scalar case $d=1$ here. 
\end{remark}

A classical technique for investigating analytical properties of \emph{initial} value problems for 
differential equations is to rewrite the given problem in the form of an equivalent
integral equation. This can be done in the fractional case in exactly the same way as in the
classical case of  first order problems, see, e.g., \cite[Lemma 6.2]{Di2010} and 
results in a Volterra integral equation. But for  \emph{terminal} value problems
of fractional differential equations this changes significantly.  
When rewriting our fractional derivative problem in  form of an integral equation, the resulting integral 
equation will be  of Fredholm type, not Volterra type \cite[Theorem 6.18]{Di2010}. In 
first order differential equations, Fredholm integral equations may arise as well, but in connection
with \emph{boundary} value problems and not with \emph{initial} value problems. Therefore 
we shall employ  techniques that are  based on  principles used for boundary value 
problems of integer order initial value problems. Specifically, 
\emph{shooting methods}, see \cite{Keller}, are the foundation of
the numerical methods that we suggest for solving fractional terminal value problems. 

To describe our method, we need further analytic prerequisites.
An important concept is the one-parameter Mittag-Leffler function 
$E_\alpha : \mathbb C \to \mathbb C$ defined for $\mathop{\mathrm{Re}} \alpha > 0$ as
\[
	E_\alpha(z) = \sum_{k=0}^\infty \frac{z^k}{\Gamma(\alpha k + 1)},
\]
see \cite{GKMR2020}. For  $\alpha$ values that are relevant in our setting we exploit the
following:

\begin{lemma}
	\label{lem:ml}
	For all $\alpha \in (0,1)$, the function $E_\alpha$ is analytic.
	Moreover, for these $\alpha$ and all $z \in \mathbb  R$ we have 
	$E_\alpha(z) > 0$ and  $E_\alpha(z)$ is strictly increasing in $z \in \mathbb R$.
\end{lemma}

\begin{proof}
	The analyticity of $E_\alpha$ follows from \cite[Proposition 3.1]{GKMR2020}. 
	The inequality $E_\alpha(z) > 0$ is trivial for $z \ge 0$ since Euler's Gamma function
	satisfies $\Gamma(w) > 0$ for all $w > 0$. For $z<0$ the inequality is a consequence
	of the properties discussed in \cite[Subsection 3.7.2]{GKMR2020} and the strict
	monotonicity follows from the properties shown in \cite[Subsection 3.7.2]{GKMR2020} as well.
	\qed
\end{proof}

To become familiar with the nature of our numerical method, it is essential to understand further properties
of initial value problems for fractional differential equations. Specifically, we mention 
the following:
Given two solutions of the same fractional differential equation on the same interval, but starting from 
two different initial values,
what are the differences between these two solutions on
this interval? Upper bounds for the differences are directly available  
by standard classical Gron\-wall type arguments. For our purposes, however, we also need
lower bounds, about which much less is known.
First we explain our
result for a linear differential equation. This is very simple but immediately
gives us important insights.

\begin{theorem}
	\label{thm:bounds-lin}
	Let  $\ell \in C[a,b]$ be given,  and
	let $y_1$ and $y_2$, respectively, be the solutions of the initial value problems
	\[
		D_a^\alpha y_k(t) = \ell(t) y_k(t), \qquad y_k(a) = y_{0,k} \qquad (k = 1,2)
	\]
	with $y_{0,1} > y_{0,2}$. Then, for all $t \in [a,b]$,
	\begin{equation}
		\label{eq:diff-bounds1}
		( y_{0,1}  - y_{0,2} ) E_\alpha (\ell_*(t) (t-a)^\alpha) 
		\le  y_1(t) - y_2(t) 
			\le ( y_{0,1} - y_{0,2} ) E_\alpha (\ell^*(t) (t-a)^\alpha)
	\end{equation}
	where $\ell_*(t) = \min_{s \in [a,t]} \ell(s)$ and $\ell^*(t) = \max_{s \in [a,t]} \ell(s)$.
\end{theorem}

\begin{proof}
	The upper bound has been derived in \cite[Theorem 5]{DT2022} and
	the lower bound is shown in \cite[Theorem 4]{DT2022}.
	\qed
\end{proof}

In the general (nonlinear) case the result is  more involved 
but the essential properties of the linear case remain intact.

\begin{theorem}
	\label{thm:bounds-nonlin}
	Assume that $f$ satisfies the hypotheses of Theorem \ref{thm:tvp-sol-ex}.
	Let $y_1$ and $y_2$, respectively, be the solutions of the initial value problems
	\[
		D_a^\alpha y_k(t) = f(t, y_k(t)), \qquad y_k(a) = y_{0,k} \qquad (k = 1,2)
	\]
	where $y_{0,1} > y_{0,2}$. Then, for all $t \in [a,b]$ we have
	\begin{equation}
		\label{eq:diff-bounds2}
		( y_{0,1} - y_{0,2} ) E_\alpha (\tilde \ell_*(t) (t-a)^\alpha) 
		\le   y_1(t)  - y_2(t) 
			\le ( y_{0,1} - y_{0,2} ) E_\alpha (\tilde \ell^*(t) (t-a)^\alpha)
	\end{equation}
	where
	\begin{subequations}
		\label{eq:ltildestar}
	\begin{equation}
		\tilde \ell_*(t) 
		= \inf_{s \in [a,t], y \ne 0} \frac{f(s, y + y_1(s)) - f(s, y_1(s))} y 
		< \infty 
	\end{equation}
	and 
	\begin{equation}
		\tilde \ell^*(t) 
		= \sup_{s \in [a,t], y \ne 0} \frac{f(s, y + y_1(s)) - f(s, y_1(s))} y 
		< \infty .
	\end{equation}
	\end{subequations}
\end{theorem}

\begin{proof}
	This is the result of \cite[Theorem 7]{DT2022}.
	\qed
\end{proof}

Theorem \ref{thm:bounds-nonlin} can be applied to the linear case  of Theorem \ref{thm:bounds-lin}. 
In this situation, the 
functions $\ell_*$ and $\ell^*$ 
of Theorem \ref{thm:bounds-lin} coincide with the functions $\tilde \ell_*$
and $\tilde \ell^*$, respectively, of Theorem \ref{thm:bounds-nonlin}.

We use the notation $\beta \sim \gamma$ for  expressions $\beta$ and $\gamma$
that depend on the same quantities to denote that there exist absolute constants $C_1 > 0$ and $C_2 > 0$ 
such that  $C_1 \beta \le \gamma \le C_2 \beta$ for all admissible values of the quantities
that $\beta$ and $\gamma$ depend on.
With this notation and the findings of Lemma \ref{lem:ml}, we can summarize the
statements of  Theorems \ref{thm:bounds-lin} and \ref{thm:bounds-nonlin} more  compactly.

\begin{corollary}
	\label{cor:bounds}
	Under the assumptions of Theorem \ref{thm:bounds-lin} or 
	Theorem \ref{thm:bounds-nonlin} and for any $y_{0,1} > y_{0,2} \in \mathbb R$ we have
	\begin{equation}
		\label{eq:bounds}
		c_* \left( y_{0,1} - y_{0,2} \right) \le y_1(b) - y_2(b) \le c^* \left( y_{0,1} - y_{0,2} \right)
	\end{equation}
	where
	\begin{equation}
		\label{eq:propconst}
		c_* =  E_\alpha (\tilde \ell_*(b) (b-a)^\alpha)
		\quad \mbox{ and } \quad
		c^* =  E_\alpha (\tilde \ell^*(b) (b-a)^\alpha)
	\end{equation}
	for the functions $\tilde \ell_*$ and $\tilde \ell^*$ of \eqref{eq:ltildestar}.
	Since $c^* > c_* > 0$ by \eqref{eq:propconst},
	we can rewrite equation \eqref{eq:bounds} as
	\begin{equation}
		y_1(b) - y_2(b) \sim y_{0,1} - y_{0,2}.
	\end{equation}
\end{corollary}

This observation is the foundation for our  numerical
method in Subsection \ref{sec:nextguess}. 

For later reference, we note a few more facts: 

\begin{remark}
	\label{rem:propconst}
	In general, we cannot expect that the ratio
	\[
	 	\hat c :=  \frac{y_1(b) - y_2(b)}{ y_{0,1} - y_{0,2}},
	\]
	i.e.\ the proportionality factor between the terminal and the initial value of the
	solution to a problem, is known exactly. However, from
	eq.~\eqref{eq:bounds}, we know that $\hat c$  is bounded  above by $c^*$ 
	and  below by $c_*$ given in  \eqref{eq:propconst}.
	As long as no additional information is available to
	obtain a more precise approximate value for $\hat c$, one may  use the mean 
	of the upper and the lower bound, i.e.\ the approximation
	\begin{equation}
		\label{eq:propconstapprox}
		\hat c \approx \frac{c_* + c^*} 2.
	\end{equation}
	To compute this value from \eqref{eq:propconst}, 
	we need to evaluate the quantities $\tilde \ell_*(b)$ and $\tilde \ell^*(b)$
	as defined in eq.~\eqref{eq:ltildestar}. If an approximate solution $\hat y$ to the
	differential equation in question is known at least for some grid points 
	$a = t_0 < t_1 < t_2 < \ldots < t_N = b$, one may select a step size $H>0$ and
	an integer  $M > 0$ and approximate  these upper and lower  bounds 
	by 
	\begin{subequations}
		\label{eq:capprox}
	\begin{eqnarray}
		\tilde \ell_*(b) & \approx & \min \Big \{ \frac{f(t_j, k H + \hat y(t_j)) - f(t_j,\hat y(t_j))}{k H} : \\
				\nonumber
				& & \qquad \qquad j \in \{ 0, 1, 2, \ldots, N \}, 
									k \in \{ \pm1, \pm 2 , \ldots, \pm M \} \Big \}
	\end{eqnarray}
	and
	\begin{eqnarray}
		\tilde \ell^*(b) & \approx & \max \Big \{ \frac{f(t_j, k H + \hat y(t_j)) - f(t_j,\hat y(t_j))}{k H} : \\
				\nonumber
				& & \qquad \qquad j \in \{ 0, 1, 2, \ldots, N \}, 
									k \in \{ \pm1, \pm 2 , \ldots, \pm M \} \Big \},
	\end{eqnarray}
	\end{subequations}
	respectively. Then we obtain $\hat c$ from these approximate values instead of the exact values $\tilde \ell_*(b)$
	and $\tilde \ell^*(b)$.
\end{remark}

\begin{remark}
	\label{rem:dissip}
	Estimating the proportionality factor $\hat c$ as 
	indicated in Remark \ref{rem:propconst} is quite useful when the differential equation in
	eq.~\eqref{eq:tvp} is dissipative, i.e.\ when $(f(t, y_1) - f(t, y_2)) (y_1 - y_2) \le 0$ for all
	$t \in [a,b]$ and all $y_1, y_2 \in \mathbb R$. In this case, eq.~\eqref{eq:ltildestar} implies that
	$\tilde \ell_*(t) \le \tilde \ell^*(t) \le 0$ for all $t$. And due to  the monotonicity 
	of the Mittag-Leffler function $E_\alpha$ (see Lemma \ref{lem:ml}), we obtain 
	\[
		0 < c_* = E_\alpha (\tilde \ell_*(b) (b-a)^\alpha)  \le c^* 
			= E_\alpha (\tilde \ell^*(b) (b-a)^\alpha)  \le E_\alpha (0) = 1 .
	\] 
	Therefore, 
	the interval $[c_*, c^*]$ in which  the correct value of $\hat c$ must lie is quite small. 
	By choosing this interval's midpoint as starting point we only make a small error.
	
	If, on the other hand,
	the differential equation is not dissipative then $c^*$ may be very much larger than $c_*$,
	and the strategy described in Remark \ref{rem:propconst} may lead to an estimate for
	$\hat c$ that is very far away from the correct value. 
	
	In Section \ref{sec:num} we work through  example problems for either case.
\end{remark}

\begin{remark}
	\label{rem:propconstauto}
	Following Remark \ref{rem:dissip} we suggest yet another method for
	approximating  $\hat c$:
	Analyze the given differential equation and see whether
	the approach of Remark \ref{rem:propconst} is appropriate (i.e., whether this does
	not lead to an excessively large value for $\hat c$). If 
	Remark \ref{rem:propconst} does not yield a useful  $\hat c$ value, choose a smaller one.
	More precisely:
	\begin{enumerate}
	\item Approximately compute the values $\tilde \ell_*(b)$ and $\tilde \ell^*(b)$ as indicated in  
		Remark \ref{rem:propconst}.
	\item If $\tilde \ell_*(b) \le \tilde \ell^*(b) \le 0$ then there is no danger of obtaining extremely large values
		for $c_*$ or $c^*$. Thus proceed as suggested in Remark \ref{rem:propconst}.
	\item If $\tilde \ell_*(b) \le 0 < \tilde \ell^*(b)$ then $c_* \le 1$, but $c^*$ may be 
		very much larger than $1$. To dampen the possible overestimation
		that $c^*$ might induce, ignore the precise value of $c^*$ and set $\hat c = 1$.
	\item If $0 < \tilde \ell_*(b) \le \tilde \ell^*(b)$ then $c^* \ge c_* > 1$. 
		Again, to mitigate  an overestimation, use the lower bound of the
		interval $[c_*, c^*]$ as an estimate for $\hat c$, i.e.\ set 
		$\hat c = E_\alpha (\tilde \ell_*(b) (b-a)^\alpha)$ 
		as suggested by the first relation in eq.~\eqref{eq:propconst}.
	\end{enumerate}
\end{remark}

\section{Description of the Method}
\label{sec:method}

\subsection{General Framework}

As indicated in Section \ref{sec:analysis}, the essential characteristics of  
problem \eqref{eq:tvp}  are identical  to those of  classical
boundary value problems and our approach involves shooting methods
\cite{Keller} which are a well established technique for boundary value problems.\\
The basic steps of general shooting methods are
as follows:
\begin{enumerate}
\item \label{step:initialguess}
	Set $k=0$. Given  problem \eqref{eq:tvp}, make an initial guess $\tilde y_0^{(0)}$ for the
	value $y(a)$.
\item \label{step:loop}
	Numerically compute the solution $\tilde y_k$ of the differential equation \eqref{eq:tvp} 
	for the initial condition $\tilde y_k(a) = \tilde y_0^{(k)}$.
	Ignore the terminal condition in \eqref{eq:tvp} in this process.
\item Compare the computed solution $\tilde y_k(b)$ with the desired  solution $y^*$ at point~$b$:
	\begin{enumerate}
	\item If $\tilde y_k(b)$ is sufficiently close to the desired solution $y^*$, 
		accept $\tilde y_k$ as the numerical
		solution of the given terminal value problem \eqref{eq:tvp} and stop.
	\item \label{step:newguess}
		Otherwise, iterate from $k$ to $k+1$ and construct a new improved guess $\tilde y_0^{(k+1)}$ for the
		starting value $y(a)$. Go back to step \ref{step:loop}.
	\end{enumerate}
\end{enumerate}

These simple components  of shooting methods will be specified more precisely 
in the subsequent subsections. The overarching concern here is to keep the
chosen shooting  algorithm's computational complexity low.  
The computational cost of shooting algorithms is  reflected in the number of operations required per 
iteration  step,  multiplied by the  number of iterations needed to achieve satisfactory accuracy.

\subsection{Selecting the Initial Guess $\tilde y_0^{(0)}$ for $y(a)$}
\label{sec:initialguess}

Unless specific  information about the given fractional ODE problem is available that suggests otherwise,
we choose $\tilde y_0^{(0)} = y^*$ as our initial guess
for $y(a)$ required in step~\ref{step:initialguess}, 
i.e.,\ we use the desired terminal value as a first guess for our initial value.

It is often assumed that a good choice of an initial guess 
at $a$ leads to  quick convergence  with
acceptable accuracy  at  $b$, while a poorly chosen initial guess 
might require more iterations, thus leading to  a significantly higher overall computational cost. 
The examples in Section \ref{sec:num}, however, indicate
otherwise. In every test example that we have considered, 
satisfactory accuracy was achieved with very few iterations with our method, no matter how close the 
starting guess at $a$  was  to the exact solution.

\subsection{Numerically Solving a Fractional ODE Initial Value Problem}
\label{subs:ivp-solver}

Algorithms that compute the solution of an (artificially constructed)
initial value problem in step~\ref{step:loop} have been discussed by 
Ford et al.\ \cite{FM2011}, showing  that
the fractional Adams-Bashforth-Moulton (ABM) method \cite{DFF2002,DFF2004} is  a good choice for non-stiff ODEs. However,
for stiff differential equations or when the interval $[a,b]$ is very large, the stability
properties of ABM may be insufficient \cite{Ga2010} and one should use an
implicit linear multistep method such as the fractional trapezoidal method \cite{Ga2015} or a 
fractional backward differentiation formula \cite{Lu1985,Lu1986}.
For our examples in Section \ref{sec:num}, we present the results
obtained with both alternative methods for comparison.
There we use a uniform discretization of the basis interval $[a,b]$ by choosing  an  integer $N> 1$ and  equally spaced grid points 
$t_j = a + j h$ for the step size  $h = (b-a) / N$. This allows us to use FFT techniques and obtain a fast
implementation, see \cite{Ga2018,HLS1985}.
Using the FFT  to compute the numerical solution on  $[a,b]$ takes  $O(N (\log N)^2)$  operations
instead of $O(N^2)$ operations for the standard implementation.

Solutions to fractional differential equations of the type considered here  
are almost never differentiable at their initial point $a$, see \cite[Theorem~6.26]{Di2010}.
This  adversely affects the convergence rate for many numerical methods
such as the ABM method, see \cite{DFF2004}. To improve convergence,
one could replace a uniform mesh by a graded one, see \cite{ZS2022}.
Or one could use a non-polynomial collocation scheme that was suggested, analyzed and tested
in \cite{FMR2014}.  Such techniques will lead to faster convergence and can reach the required accuracy with  
larger step sizes and lower the overall 
computational effort. But these ideas cannot be easily
combined with the FFT technique and consequently their numerical schemes become more costly overall.
We shall not pursue these latter approaches  any further.

\subsection{Improved Subsequent Guesses for the Initial Values}
\label{sec:nextguess}

The major contribution of our work is a new  efficient method for guessing the initial values $y(a)$ that hits
$y^* = y(b)$ more and more accurately. Traditional approaches 
\cite{Di2008,Di2015,FMR2014}  use
 classical bisection that halves the size of the containment interval for the ``correct'' choice of
$y(a)$  in each step. Clearly bisection is convergent, but it takes a large number of iterations to arrive in 
a sufficiently small neighbourhood of the exact solution if the size
of the containment interval is large such as $10^7$ units wide. 
Note that ten interval halving steps reduce the error of the initial value location only by a factor of about $10^3$ since 
$2^{10} \approx$ 1,000 and it would take 40 guess iterations to reduce this error to a reasonable $10^{-5}$. We suggest a different
method that converges much faster. Section \ref{sec:num} compares the new approach with classical bisection based methods.

Like the classical bisection method, our approach also requires two initial guesses
for the initial values $\tilde y_0^{(0)}$
and $\tilde y_0^{(1)}$. For $\tilde y_0^{(0)}$  we always choose $\tilde y_0^{(0)} := y^*$, the given 
terminal value. The next guess for a starting value and all subsequent guesses
are  chosen according to Theorems \ref{thm:bounds-lin} 
and \ref{thm:bounds-nonlin} and the fact that any two solution curves of a given fractional ODE with 
different initial values cannot cross each other. Hence, two solution curves for different starting 
values $y_{0,1} > y_{0,2}$ can either spread out or bunch up further over the time interval $[a, b]$. 
By Corollary \ref{cor:bounds}, the proportion of two solution values 
$y^*_1$ and $y^*_2$ obtained for $t = b$ and their starting values $y_{0,1}$ and $y_{0,2}$ at $t = a$ 
indicates how to space the initial values until we find a starting value $y_0^*$ that reaches 
the desired final value $y^*$ within a chosen error bound.

As long as only one initial guess $\tilde y_0^{(0)}$ is available, i.e., when 
the next guess $\tilde y_0^{(1)}$ has not yet  been  computed, we assume---due to a lack of any
information that might suggest otherwise---that the proportionality factor $\hat c$ between
the terminal values (i.e.\ the function values of the solution at $t=b$) and the initial values 
(the corresponding values for $t=a$), see Remark \ref{rem:propconst},
is given by eq.~\eqref{eq:propconstapprox}. The values of $c_*$ and $c^*$ in this
formula are then replaced by their approximations indicated in eq.~\eqref{eq:capprox}. 
According to  Remark \ref{rem:propconst}, our next guess for the initial value becomes
\begin{equation}
	\label{eq:nextguess1}
	\tilde y_0^{(1)} 
		:= \tilde y_0^{(0)} + \frac{ y^*  - \tilde y_0(b) }{\hat c}.
\end{equation}
$ \tilde y_0^{(1)} $ is equal to the previous guess $\tilde y_0^{(0)}$ if and only if the latter 
has resulted in the exact solution, i.e.,\ if and only if $\tilde y_0(b) = y^*$ and the problem has been solved.

\begin{remark}
	Note that  evaluating  the formulas  in \eqref{eq:capprox}
	requires  knowledge of an approximate solution to the given differential
	equation for some initial condition. At this stage, such information
	is already available because we have computed a 'solution' 
	using the first guess $\tilde y_0^{(0)}$ as initial value.
\end{remark}

\begin{remark}
	\label{rem:hatc-simple}
	\begin{itemize}
	\item The approach described in Remark \ref{rem:propconst} to compute the value $\hat c$
		requires the evaluation of the Mittag-Leffler function $E_\alpha$, cf.~eq.~\eqref{eq:propconst}.
		For this purpose, we suggest to use the algorithm developed in \cite{Ga2015b}.
	\item In case of a non-dissipative fractional differential equation, we have seen in Remark
		\ref{rem:dissip} that the approach of Remark \ref{rem:propconst}
		may lead to  very poor approximations of $\hat c$ and its true value
		may be massively over-estimated. Therefore, for non-dissipative problems, 
		one is likely to be better off with simply choosing an arbitrary not excessively large positive number
		for $\hat c$ such as\ $\hat c = 1$, in eq.~\eqref{eq:nextguess1}. 
	\item If the user believes that evaluating the formulas in \eqref{eq:capprox}
		is too expensive  then one may again use $\hat c = 1$ in equation~\eqref{eq:nextguess1}, even for dissipative  equations.
		Then the guess of ~$\tilde y_0^{(1)}$ may be  
		 worse than the one obtained with  $\hat c$ from \eqref{eq:propconstapprox}
		and the number of iterations until  satisfactory accuracy  
		may increase slightly. 
	\item Remark \ref{rem:propconstauto} suggests another way to choose $\hat c$ as it tries to find a compromise between the two earlier suggestions.
	\end{itemize}

	In the numerical experiments of Section \ref{sec:num}, we  report the results of our new method 
	for various choices of $\hat c$ such as  the construct of Remark \ref{rem:propconst},
	the idea of Remark \ref{rem:propconstauto}, or simply setting $\hat c = 1$. In practice
	the method of Remark \ref{rem:propconstauto} usually leads to the smallest number of iterations and quickest overall convergence.
\end{remark}

Our strategy for constructing additional initial values $\tilde y_0^{(k)}$ for $k = 2, 3, \ldots$ is  based on
Corollary \ref{cor:bounds} which states that,
given two fractional initial value problems for the same  fractional differential equation, but starting from  
different initial conditions, the difference in  terminal values of these two problems 
is approximately proportional to the difference in their initial values. Strictly speaking, Corollary \ref{cor:bounds} only holds in the asymptotic 
sense when the difference between subsequent initial values is tending to zero. 
In practice, our \emph{proportional secting} idea has worked exceedingly well for all types of fractional ODEs
even if this assumption is not satisfied,  and the  results  in Section~\ref{sec:num} show  this clearly.

Next we need to specify  initial value
guesses for $\tilde y_0^{(k)}$ when $k \ge 2$ from two earlier guessed starts. Our algorithm  always analyzes the two most recent  
iteration results for $y$ at $b$ and compares their calculated approximations
$\tilde y_{k-1}(b)$ and $\tilde y_{k-2}(b)$ with the desired value $y^*$.    
How are these three values for $y$ at $b$  positioned? For this we found it convenient to express 
the target value $y^*$ as a generalized convex combination of the two $\tilde y_{\mu}(b)$ values for
$\mu \in \{k-2, k-1\}$ and write
\begin{equation}
	\label{eq:convcomb-target}
	y^* = \lambda_k  \tilde y_{k-1}(b) + (1 - \lambda_k) \tilde y_{k-2}(b)
\end{equation}
for some  $\lambda_k \in \mathbb R$. Here $\lambda_k$ can be immediately  computed 
since all other quantities  in  \eqref{eq:convcomb-target}
are known. 
Since any positioning of the three values relative to each other is possible, this concept  may lead to $\lambda_k < 0$
or $\lambda_k > 1$ which would not be admissible in a classical convex combination, but this does not create any difficulties for our algorithm. 
From  $\lambda_k$ we compute the
new guess for the next shooting start as 
\begin{equation}
	\label{eq:nextguess2+}
	\tilde y_0^{(k)} := \lambda_k \tilde y_0^{(k-1)}+ (1 - \lambda_k) \tilde y_0^{(k-2)}.
\end{equation}
The new starting guess $\tilde y_0^{(k)}$ is a generalized convex combination of the two preceding guesses that uses  the same proportions as those in eq.~\eqref{eq:convcomb-target}. 
Evidently, if the  statement of Corollary \ref{cor:bounds} were an equality and no
errors had occurred in the numerical solver for the initial value problem, then \eqref{eq:nextguess2+} would lead to a starting guess that  hits the desired target value $y^*$ exactly. 

With the value of $\lambda_k$ in eq.~\eqref{eq:nextguess2+} as computed 
from equation~\eqref{eq:convcomb-target},
we  obtain the next initial value  guess as 
\begin{equation}
	\label{eq:c3}
	\tilde y_0^{(k)} 
		= \tilde y_0^{(k-1)} 
			+ \left (y^* - \tilde y_{k-1}(b) \right) 
				\frac{\tilde y_0^{(k-1)} - \tilde y_0^{(k-2)}}{\tilde y_{k-1}(b) - \tilde y_{k-2}(b)} .
\end{equation}
Thus, the correction term that gets us from the previous initial value 
$\tilde y_0^{(k-1)}$ to the next starting guess $\tilde y_0^{(k)}$ is proportional to the error of the previous terminal value $y^* - \tilde y_{k-1}(b)$ multiplied by 
 the proportionality factor made up of  the quotient of the difference of the two preceding initial values
and  the difference of the two preceding terminal values.

Note that formula \eqref{eq:c3} for $\tilde y_0^{(k)}$  is independent of the actual lay of $\tilde y_0^{(k-1)}$ and $\tilde y_0^{(k-2)}$ with respect to each other or to $y^*$. This formulation was chosen deliberately to avoid any lay-logical tree complications when executing our \emph{proportional secting method}. 
The reference point in eq.~\eqref{eq:c3} is always $y^*$. 
The algorithm  computes $\tilde y_0^{(k)}$ whose 
associated function value $y_k(b)$ 
is closer to $y^*$ than at least one of those generated by the initial values $\tilde y_0^{(k-1)}$ or $\tilde y_0^{(k-2)}$. Once $\tilde y_0^{(k)}$ and the associated terminal value $\tilde y_k(b)$ 
have been computed, we drop the oldest point data pair $\tilde y_0^{(k-2)}$ 
and $\tilde y_{k-2}(b)$ and continue with the pair with indices $k$ and $k-1$ 
and iterate on until 
$|y_k(b) - y^*|$ has dropped below the required accuracy threshold.

Compared to  classical bisection, our approach has two significant advantages:
\begin{enumerate}
\item Before a classical bisection method can be started, the correct 
	initial value $y(a)$ of the solution $y$ must be known to lie inside the 
	search interval, i.e., two numbers
	$\underline{y_0}$ and $\overline{y_0}$ with $y(a)
	\in [\underline{y_0}, \overline{y_0}]$ must have been computed for the actual solution $y$  with $y(b) = y^*$. 
	Any first guess $y_0^{(0)}$ provides 
	one of the search interval bounds, but to find another that lies on the opposite side of the 
	unknown exact value of $y(a)$, 
	further iterations are necessary. The proportional secting method does not require any of this; 
	to compute $\tilde y_0^{(k+1)}$ we do  not need to know anything about the lay of 
	$y(a)$, $\tilde y_0^{(k)}$, or $\tilde y_0^{(k-1)}$.
\item In  classical bisection, one starts with the initial interval $[\underline{y_0}, \overline{y_0}]$
	in which the exact solution's value for $y(a)$  is known to be located. In each iteration step, the size 
	of this interval (and hence the accuracy with which one knows the correct initial
	value) is reduced by one half. While this method clearly converges, it is easy to see
	that its convergence is typically rather slow. When  
	the interval $[\underline{y_0}, \overline{y_0}]$ is large, classical bisection  often 
	requires very  many iterations for an acceptable accuracy in the $10^{-6}$ or $10^{-8}$ range. 
	Our examples 
	demonstrate that our proportional secting scheme reduces the size of the
	initial search interval much faster and thereby it solves the shooting problem with fewer  iterations.
\end{enumerate}

\begin{remark}
	Searching for the correct initial value is a nonlinear equations problem.  The proportional secting method 
	solves this nonlinear equation by the secant method. Such an approach
	for handling fractional terminal value problems  was briefly mentioned by Ford and
	Morgado \cite[Section~3]{FM2011}. However, their focus  was on the  
	selection of IVP solvers and not on  shooting strategies. The
	authors of \cite{FM2011} have neither stated any properties
	of this approach nor provided an analysis or given any reasons why to use this method.
	The two main advantages of the secant method over the bisecetion approach
	seem to have been unnoticed so far.
\end{remark}

\begin{remark}
	Our approach 
	 always replaces the older of the two previous initial values, viz.\ $\tilde y_0^{(k-2)}$,
	by the newly computed value $\tilde y_0^{(k)}$ and then proceeds to the next iteration with the
	pair $(\tilde y_0^{(k-1)}, \tilde y_0^{(k)})$.
	If $y^* = y(b)$ lies inside of the interval bounded by $\tilde y^{(k-1)}(b)$ and 
	$\tilde y^{(k-2)}(b)$  we have obtained a guaranteed enclosure
	\[
		y(a) \in \left [\min \left \{ \tilde y^{(k-1)}_0, \tilde y^{(k-2)}_0 \right \}, 
					\max \left \{ \tilde y^{(k-1)}_0, \tilde y^{(k-2)}_0 \right \} \right ]
	\]
	of the solution's exact initial value. Our algorithm does not guarantee such an enclosure in any iteration and we never know whether 
	\[
		y(a) \in \left [ \min \left \{ \tilde y^{(k)}_0, \tilde y^{(k-1)}_0 \right \}, 
					\max \left \{ \tilde y^{(k)}_0, \tilde y^{(k-1)}_0 \right \} \right] .
	\]
	Since for most practical applications it is not relevant to have this information, we actually do not consider this 
	a major drawback.
	We could modify proportional secting so that 
	it retains the enclosure property once it has been obtained. 
	Instead of always replacing the older of the two previous values, we might just replace the 
	value that is on the same side of the exact solution as the new one, thus employing the
	\emph{regula falsi} (false position method). However, using the regula falsi here
	comes with the added cost of potentially requiring  significantly more iterations before reaching
	acceptable accuracy. Therefore we do not pursue this idea further.
\end{remark}

Figure \ref{fig:alg} visualizes the proportional secting iterations for the terminal value problem 
\[
	D_0^\alpha y(t) = \frac 1 {t+1} \sin (t \cdot y(t)) \
	\text{ with } \ y(20) = y^* = 0.8360565
\]
that we will discuss in more detail in Example \ref{ex8} below.
Our interval of interest is $[a,b] = [0, 20]$.
We start with $\tilde y_0^{(0)} = y^* \approx 0.836$ at $a = 0$, use the BDF2 solver
for the initial value problem and obtain the black approximate
solution graph that arrives at $\tilde y_0(b) \approx 0.57$ at time $b = 20$. 
For  simplicity's sake we follow  Remark \ref{rem:hatc-simple}
and construct the next initial value $\tilde y_0^{(1)}$ based on formula \eqref{eq:nextguess1} 
with $\hat c = 1$.
Since $\tilde y_0(b) \approx 0.58 < 0.836\ldots = y^*$, this moves the initial
value exactly $y^* - \tilde y_0(b)$ units up to obtain $\tilde y^{(1)}_0 \approx 1.1$ 
and subsequently $\tilde y_1(b) \approx 0.89$ when traveling along the blue graph. 
The next solution graph from $\tilde y^{(2)}_0 \approx 1.05$ (shown in red) arrives about halfway 
between $\tilde y_1(b)$ and $y^*$ at time~$b$. 
The dotted graph finally reaches $y^*$ at $b$ in 8 iterations with a $10^{-15}$ error. 
An absolute error of approximately $10^{-10}$ at the terminal point $b$ takes 7 iterations,
and 6 iterations suffice to obtain a $10^{-7}$ accuracy
for this example. Thus, each extra iteration
gives us 3 to 4 more accuracy digits at time $b$.

\begin{figure}[htb]
	\centering
	\includegraphics[width = 0.95\textwidth]{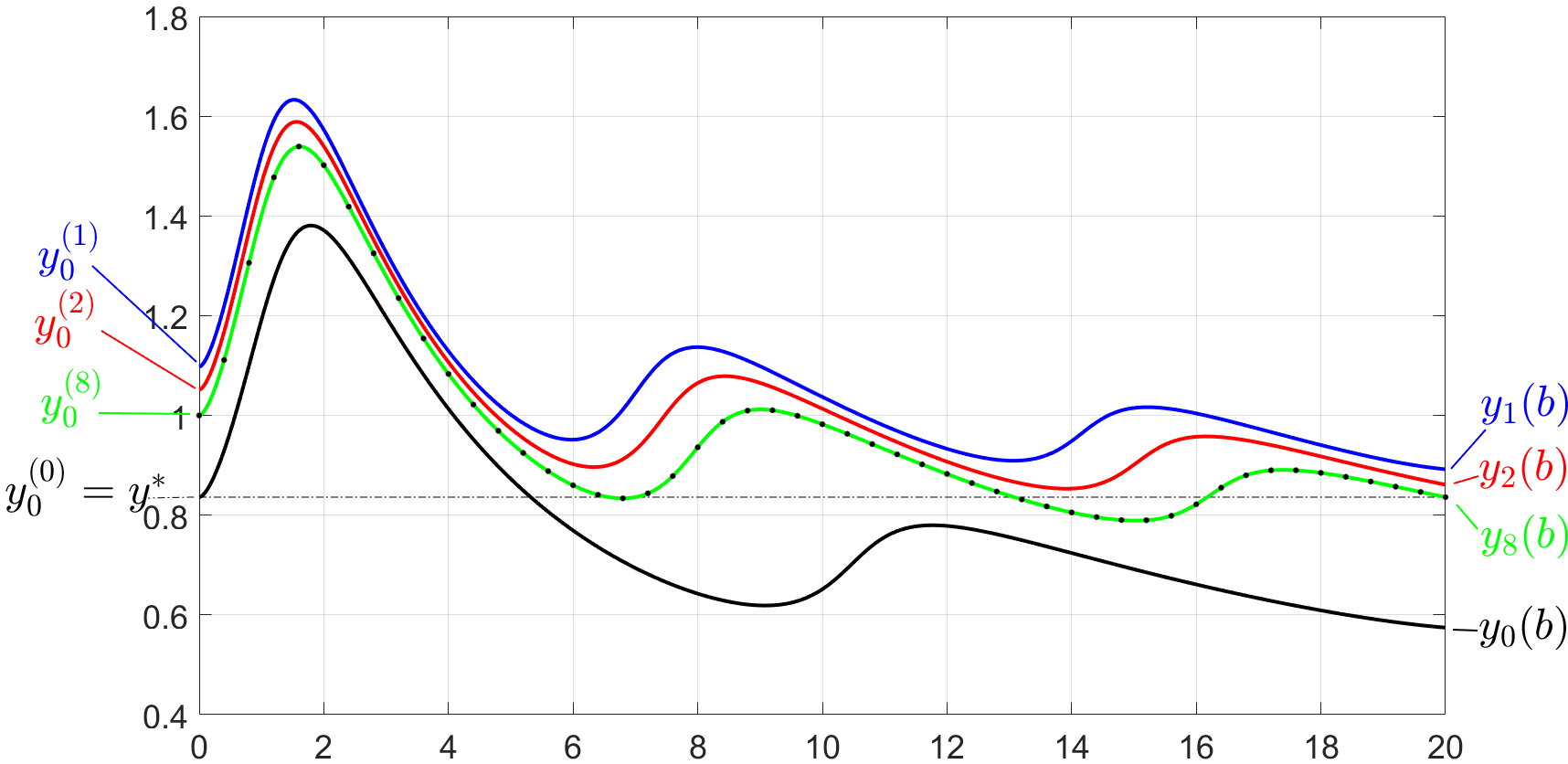}
	\caption{\label{fig:alg}Visualization of the behavior of the algorithm when applied to Example
		\ref{ex8}. The dotted curve is the numerical solution after 8 iterations
		of our shooting method; it cannot be visually distinguished from the exact solution.}
\end{figure}

\subsection{Selection of the Step Size for the Numerical IVP Solver}

To reduce the computational cost of any iterative shooting  algorithm, 
\cite{Di2015} has proposed to vary the step size of the algorithm. When the iteration counter $k$ of the ``shots''  is
small, one assumingly is relatively far  
from the exact solution because the initial value is not sufficiently accurate and it does not make sense to solve the initial value problem with high accuracy.  Moreover one can  likewise use  relatively large steps in the early iterations. Our numerical experiments
in Section~\ref{sec:num} indicate that no such difficult to implement  step size varying procedure
is necessary for either one of the three variants of our proportional secting approach because 
we always arrive at a very accurate solution in a small number (3 to 8 instead of  18 to 65) of  iterations.
Therefore we have decided to use a fixed step size  and the same IVP solver in all shooting iterations.

\subsection{Algorithmic Description of the Proportional Secting Scheme}

We now write out the proportional secting method  in a formal pseudo-code like manner in Algorithm \ref{alg:propsec} below.

\begin{algorithm}
	\renewcommand{\algorithmicrequire}{\textbf{input data:\ \ \ }}

	\begin{algorithmic}
	Shooting method based on proportional secting\label{alg:propsec}
	
	\REQUIRE{
	\begin{itemize}
	\item the terminal value problem \eqref{eq:tvp}, i.e.,\ the function $f$ on the
		right-hand side of the differential equation and the desired terminal value $y^*$
	\item a numerical method for solving fractional initial value problems
		(e.g., the fractional Adams method, the fractional trapezoidal method 
			or a fractional BDF)
	\item a step size $h$ to be used by the numerical IVP solver
		(or an arbitrary set of mesh points for the
		discretization of the time interval $[a,b]$)
	\item a desired accuracy level $\varepsilon > 0$\smallskip
	\item a decision strategy for choosing the value $\hat c$ in eq.~\eqref{eq:nextguess1}
		with three options: the generic choice $\hat c = 1$ and the procedures indicated in 
		Remarks \ref{rem:propconst} and \ref{rem:propconstauto}, respectively
	\end{itemize}
	}
	\hrule
	\smallskip

	\STATE{set $k := -1$}
	\REPEAT
		\STATE increment $k$ by $1$
		\IF{$k=0$}
			\STATE set $\tilde y_0^{(k)} := y^*$
		\ELSIF{$k=1$}
			\STATE set $\tilde y_0^{(k)}$ according to formula \eqref{eq:nextguess1}
		\ELSE
			\STATE set $\tilde y_0^{(k)}$ according to formula \eqref{eq:c3}
		\ENDIF
		\STATE Solve the initial value problem $D_a^\alpha \tilde y_k(t) = f(t, \tilde y_k(t))$, 
			$\tilde y_k(a) = \tilde y_0^{(k)}$
			with the given numerical method and the chosen step size $h$
			or the given mesh, respectively.
	\UNTIL{$|y^* - \tilde y_{k}(b) | \leq \varepsilon$}
	\RETURN the numerical solution $\tilde y_k$
	\end{algorithmic}
\end{algorithm}

\section{Analysis of the Algorithm}

For a better understanding of the performance and  behavior of our algorithm, we  now give a short theoretical analysis.

\subsection{Accuracy and Convergence}

We begin with a look at 
error estimates. The total error 
of our numerical scheme has two active components originating from two different 
sources besides rounding and truncation errors. 

When solving the terminal value problem \eqref{eq:tvp} in the $k$-th iteration of the shooting procedure, 
we solve an adjacent problem from the starting guess 
$y(a) = \tilde y_0^{(k)}$ and obtain $y_k$ with $y_k(b) \neq y(b) = y^*$ for the exact solution $y$. 
This is one component $e_1(k)$ of the total error. Additionally there is another error component 
$e_2(k,N)$ associated with the  $k$-th iterative solution $y_k$, namely the computational error
 $e_2(k, N)$ inherent in the numerical integration algorithm that we use.
$e_2(k, N)$ depends
on the number $N$ of grid points of the fractional IVP solver and on the 
$k$-th initial value problem. 
We shall assume that the grid points are uniformly spaced as $t_j = a + j (b-a) / N$  for $j = 0, 1, 2, \ldots, N$.

We first consider the limit case $k \to \infty$, i.e.\ we assume that we have done
so many shooting iterations that the exact terminal value $y_\infty(b) = y(b) = y^*$ 
is reached numerically. We denote the solutions
for this special terminal value problem by $y_\infty$ for the exact solution
and by $\tilde y_{\infty, N}$ for the numerical solution 
where, in contrast to the convention used in
the remainder of this paper, the notation explicitly indicates the number $N$ of grid points
used in the IVP solver. Since   $k$
cannot be varied any more in the limit case, it remains to study the influence of the
parameter $N$, i.e.,\ the number of grid points of the IVP solver, on the error.
The following theorem  indicates that the convergence order of the 
underlying IVP solver is retained.

\begin{theorem}
	\label{thm:err-bound-infty}
	Assume the hypotheses of Theorem \ref{thm:tvp-sol-ex} and
	that the IVP solver approximates the solution of the
	given fractional differential equation in eq.~\eqref{eq:tvp} with 
	an $O(N^{-p})$ convergence order with some constant $p > 0$ 
	for any initial condition. Then 
	\begin{equation}
		\label{eq:err-k-infty}
		\max_{j = 0, 1, 2, \ldots, N} | y(t_j) - \tilde y_{\infty, N}(t_j) | = O(N^{-p}).
	\end{equation}
\end{theorem}

\begin{remark}
	In Theorem \ref{thm:err-bound-infty}, we require the
	initial value solver to solve the given differential equation with an 
	$O(N^{-p})$ error for \emph{every} initial value, not just for the 
	special initial value of the exact solution of the given terminal value problem. 	
	This requirement 
	is due to a special property of fractional differential equations: 
	The solutions of fractional differential equations tend to have only weak 
	smoothness properties in general, but may behave  
	much more smoothly  for some  exceptional initial values, see 
	\cite[Section 6.4]{Di2010}. If the exact solution of the terminal value problem 
	has such an exceptional initial value then the IVP solver may 
	compute a  numerical
	approximation rapidly  if starting from the exact initial 
	value. But generally the iterations converge much more slowly for all other, even  
	nearby, initial values.
\end{remark}

\begin{proof}
	Our proof has two steps. First we prove that
	\begin{equation}
		\label{eq:proof-err-k-infty-1}
		| y(t_0) - \tilde y_{\infty, N}(t_0) | = O(N^{-p})
	\end{equation}
	and then we use \eqref {eq:proof-err-k-infty-1} to show 
	eq.~\eqref{eq:err-k-infty}. 
	
	In an indirect proof, we assume 
	that \eqref{eq:proof-err-k-infty-1} is not true. Then there exists a
	function $\phi : \mathbb N \to \mathbb R$ and 
	a strictly increasing sequence 
	$(n_\mu)_{\mu=1}^\infty$ of positive integers  such that 
	$\lim_{\mu \to \infty} \phi(n_\mu) = \infty$ and
	\begin{equation}
		\label{eq:proof-err-k-infty-2}
		| y(t_0) - \tilde y_{\infty, n_\mu}(t_0) | \ge c n_\mu^{-p} \phi(n_\mu)
	\end{equation}
	for all sufficiently large $\mu$ and some constant $c > 0$. 
	By definition of $\tilde y_{\infty, N}$, we have $\tilde y_{\infty, N}(t_N) =
	\tilde y_{\infty, N}(b) = y^*$ and thus 
	\begin{equation}
		0 
		=    N^p \cdot | y^* - \tilde y_{\infty, N}(t_N) | 
		=    N^p \cdot | y^* - y_\infty(t_N) +  y_\infty (t_N) - \tilde y_{\infty, N}(t_N) | 
		\ge  |  z_{1,N} - z_{2,N} |
		\label{eq:proof-err-k-infty-3}
	\end{equation}
	where
	\[
		z_{1,N} =  N^p \cdot | y^* - y_\infty(t_N) |
		\quad \mbox{ and } \quad
		z_{2,N} =  N^p \cdot | y_\infty (t_N) - \tilde y_{\infty, N}(t_N) | .
	\]
	And  whenever $N$ occurs in the sequence $(n_\mu)$,
	we have: 
	\begin{itemize}
	\item After convergence $z_{1,N}$ equals $ N^p$ times the absolute value of the difference 
		between the exact solution $y^*$ of the given terminal value problem
		\eqref{eq:tvp} evaluated at $t=b$, i.e.\ the solution to the initial value problem for the associated 
		differential equation combined with the initial value $y(t_0)$,
		and the exact solution of
		the terminal value problem when computed by the shooting method. 
		The shooting method  terminal value problem is equivalent to the initial value problem 
		for the same differential equation but with the initial value $\tilde y_{\infty, N}(t_0)$.
		Thus, by Corollary \ref{cor:bounds} and eq.~\eqref{eq:proof-err-k-infty-2}, we obtain 
		\begin{eqnarray*}
			z_{1,N} 
			&\ge&  N^p \cdot | y(t_0) - \tilde y_{\infty, N}(t_0) | 
						\cdot E_\alpha(\tilde \ell_*(b) (b-a)^\alpha) \\
			&\ge& c N^p N^{-p} \phi(N) 
			= c \phi(N) \to \infty ~ .
		\end{eqnarray*}
	\item After convergence $z_{2,N}$ equals $N^p$ times the absolute value of the error of the 
		numerical solution to the initial value problem
		for $y_\infty$ at the point $t_N$. By assumption, the IVP solver converges at
		an $O(N^{-p})$  rate, and therefore there exists some 
		constant $c' > 0$ such that $0 \le z_{2,N} \le c' N^p N^{-p} = c'$ for all sufficiently 
		large $N$ in the sequence $(n_\mu)$.
	\end{itemize}
	Consequently, for sufficiently large $N$ in the sequence $(n_\mu)$, we have $z_{1,N} > z_{2,N}$,
	and hence the rightmost entry  of eq.~\eqref{eq:proof-err-k-infty-3} is strictly positive,
	giving the required contradiction.
	So eq.~\eqref{eq:proof-err-k-infty-1}  has been proved.
	
	Now note that
	\[
		\max_{j = 0, 1, 2, \ldots, N} | y(t_j) - \tilde y_{\infty, N}(t_j) | 
		\le d_1(N) + d_2(N)	
	\]
	where
	\begin{eqnarray*}
		d_1(N) &=& \sup_{t \in [a, b]} | y(t) - y_\infty(t) | \\
		\mbox{and } \quad
		d_2(N) &=&  \max_{j = 0, 1, 2, \ldots, N} |y_\infty(t_j) - \tilde y_{\infty, N}(t_j)|.
	\end{eqnarray*}
	And $\tilde y_{\infty,N}(t_0) = y_\infty(t_0)$
	because the initial value problem solver that generates the approximation $\tilde y_{\infty, N}$
	is exact for the initial point so that
	\begin{eqnarray*}
		d_1(N) 
		&\le& |y(t_0) - y_\infty(t_0)| \cdot E_\alpha(\tilde \ell^*(b) (b-a)^\alpha) \\
		&\le& |y(t_0) - \tilde y_{\infty,N}(t_0)| \cdot E_\alpha(\tilde \ell^*(b) (b-a)^\alpha)
		= O(N^{-p})
	\end{eqnarray*}
	by Corollary \ref{cor:bounds} and eq.~\eqref{eq:proof-err-k-infty-1}. Moreover, 
	due to  the convergence rate of the initial value problem solver,
	\begin{eqnarray*}
		d_2( N) &=& \max_{j = 0, 1, 2, \ldots, N} |y_\infty(t_j) - \tilde y_{\infty,N}(t_j)| = O(N^{-p})
	\end{eqnarray*}
	for sufficiently large $N$. 
	\qed
\end{proof}

For  only  finitely many shooting steps the following similar
error bounds hold.

\begin{theorem}
	\label{thm:err-bound-k}
	Let $\tilde y_k$ be the numerical solution of a given terminal value 
	problem \eqref{eq:tvp}
	obtained after $k \ge 2$ steps of our shooting method. Then,
	under the assumptions of Theorem \ref{thm:err-bound-infty}, we have
	\[
		\max_{j = 0, 1, 2, \ldots, N} | y(t_j) - \tilde y_k(t_j) |  
			\le e_1(k) + e_2(k, N) 
	\]
	where
	\[
		e_1(k) 
		= \sup_{t \in [a, b]} | y(t) - y_k(t) | 
		\le |y(a) - y_0^{(k)}| \cdot E_\alpha(\tilde \ell^*(b) (b-a)^\alpha)
	\]
	depends only on $k$ and
	\[
		e_2(k, N) 
		= \max_{j = 0, 1, 2, \ldots, N} |y_k(t_j) - \tilde y_k(t_j)| 
		= O(N^{-p})
	\]
	depends on $k$ and the chosen number $N$ of grid points.
\end{theorem}

\begin {proof}
	This can be shown much in the same way as Theorem \ref{thm:err-bound-infty}.
	\qed
\end{proof}

Numerical experiments in Section  \ref{sec:num} will illustrate these estimates.

\subsection{Stability and Robustness}

The numerical stability and robustness of an algorithm are  relevant issues 
when assessing its practical usefulness.  For shooting algorithms we have to deal with two essential aspects
in this context.

We need to understand 
the fundamental idea  here to solve initial value problems
that are close to, but not identical to the initial value problem that are
equivalent to the given terminal value problem.
From Theorem~\ref{thm:bounds-nonlin} and Corollary~\ref{cor:bounds},
we can see the well-posedness of both the original terminal value problem
and the associated equivalent initial value problem. Thus small changes 
in either of these problems, no matter whether they are due to the 
way in which the shooting method works, to rounding errors of the given data or  
 inherent errors of the initial value problem solver, or anything else
do not lead to significant perturbations of the algorithm's output. 

Another key component is the initial value solving
algorithm that is executed  in every
iteration of a shooting method. If  an integrator  with poor stability 
properties is used or if the chosen step size is too large to guarantee stability, then the 
instability will be propagated  into the shooting method and may render 
its output meaningless.  
Fortunately, the stability properties of many frequently used fractional IVP solvers
in the context of fractional ODEs are well understood: for the Adams-Bashforth-Moulton method (ABM), see \cite{Ga2010}, for the 
fractional linear multistep methods such as the fractional BDF2 or the fractional
trapezoidal methods, see  \cite{Lu1985}; for additional information see 
\cite{Ga2015}. These well established  IVP solving methods are highly suitable for our purposes and they will be used in the numerical examples of Section \ref{sec:num}.

\section{Numerical Results}
\label{sec:num}

Here we  present  numerical experiments with our new proportional secting scheme, first  in all its three variants of choosing  $\hat c$ required in eq.~\eqref{eq:nextguess1}
when computing the second guess for the initial value and continuing with further proportional secting  iterations.
We compare our new method  with the conventionally used shooting method that uses bisection.
Our algorithm has been implemented  in MATLAB R2022a on a notebook with an Intel Core i7-8550U
CPU clocked at 1.8~GHz running Windows 10 and in MATLAB R2022b on a  MacBook Pro with 2.4 GHz Quad-Core Intel Core i5 and 16 GB RAM.

In all cases, we have tested the shooting methods with two different solvers for the initial value problems, using 
the Adams-Bashforth-Moulton (ABM) scheme and the second order backward differentiation formula (BDF2) of \cite{Lu1985}. 
The ABM method was implemented in a P(EC)$^m$E structure with four corrector iterations
\cite{Di2003}.  BDF2 is an implicit method and hence needs to solve a nonlinear equation at each time step to compute the corresponding approximate solution. For this we use Garrappa's implementation \cite{Ga2015}. We terminate 
its Newton iterations when two successive values differ  by 
less than $10^{-10}$. All shooting iterations are  
terminated when the approximate solution $y_k$ at the endpoint $b$ of  $[a,b]$  
differs by at most $\varepsilon$ from the desired value  $y(b) = y^*$ for $\varepsilon = 10^{-6}$, $\varepsilon = 10^{-8}$, and $\varepsilon = 10^{-10}$ in our tests.
 
The tables below list the chosen initial value problem solver together with the corresponding step size,
the maximal error over the interval of interest and the number of iterations that each of the shooting methods
needed with the respective combination of IVP solver and step size to converge up to the required accuracy. 
In this context, we note that, since the shooting strategies---and hence the sequences of the chosen 
initial values---differ from each other, the approximate solutions computed by the four different approaches
do not coincide exactly. Therefore, the respective maximal errors are also not precisely identical.
However,  at least for $\varepsilon = 10^{-8}$ and
$\varepsilon = 10^{-10}$, the maximal errors agree with each other at least within the accuracy 
listed in the tables. For $\varepsilon = 10^{-6}$, the error variations are somewhat larger but their values have a common order of magnitude. 
For this case the corresponding table columns list the errors for the worst of our four guessing approaches.

\begin{example}
	\label{ex3}
	Our first example is the terminal value problem
	\begin{eqnarray*}
		D_0^\alpha y(t)
		&  = & \frac{8! \cdot  t^{8-\alpha} }{\Gamma(9-\alpha)}
			- \frac{3 \Gamma(5+\alpha/2) t^{4-\alpha/2}}{\Gamma(5-\alpha/2)} \\
		& & {}  + \frac 9 4 \Gamma(1+\alpha) 
				+ \left(\frac 3 2 t^{\alpha/2} - t^4 \right)^3 - |y(t)|^{3/2}, \\
		y(1) &  = & \frac 1 4,
	\end{eqnarray*}
	whose exact solution is 
	\[
		y(t) = t^8 - 3 t^{4+\alpha/2} + \frac 9 4 t^\alpha.
	\]
\end{example}
	
This is a standard example used for testing numerical methods in fractional
calculus; cf., e.g., \cite{DFF2002}.
We report the results for the special case $\alpha = 0.3$ in Tables
\ref{tab:ex3a}, \ref{tab:ex3b} and \ref{tab:ex3c} below.

\begin{table}[!htb]
	\footnotesize
	\caption{\label{tab:ex3a}Computational cost and accuracy obtained when
		solving Example \ref{ex3} for $\alpha = 0.3$ with different numerical
		methods and $\varepsilon = 10^{-6}$.}
	\centering
	\begin{tabular}{l|c|c|c|c|c|c}
  & & & \multicolumn{4}{c}{\bf number of shooting iterations \ \ \ \ \ \ } \\
	    &      &      &  {\bf bisection}   & \multicolumn{3}{c}{\bf proportional secting} \\
	    &      &      &  &  using             & choosing $\hat c$ in & choosing $\hat c$ in \\
	IVP & step & max. &           &  $\hat c = 1$ & \eqref{eq:nextguess1} according & \eqref{eq:nextguess1} according \\
 solver & size & error &  & in \eqref{eq:nextguess1} & to Remark \ref{rem:propconst}  & to Remark \ref{rem:propconstauto} \\
		\hline
\multirow[c]{7}{*}{Adams}    &  0.00200000  & $ 6.7 \cdot 10^{-6^{\vphantom{1}}}$  & 37 &  5 &  5 &  5 \\
    &  0.00100000  & $ 2.9 \cdot 10^{-6^{\vphantom{1}}}$  & 38 &  5 &  5 &  5 \\
    &  0.00050000  & $ 9.5 \cdot 10^{-7^{\vphantom{1}}}$  & 19 &  5 &  5 &  5 \\
    &  0.00025000  & $ 1.9 \cdot 10^{-6^{\vphantom{1}}}$  & 18 &  5 &  5 &  5 \\
    &  0.00012500  & $ 1.9 \cdot 10^{-6^{\vphantom{1}}}$  & 18 &  5 &  5 &  5 \\
    &  0.00006250  & $ 1.9 \cdot 10^{-6^{\vphantom{1}}}$  & 18 &  5 &  5 &  5 \\
    &  0.00003125  & $ 1.9 \cdot 10^{-6^{\vphantom{1}}}$  & 18 &  5 &  5 &  5 \\
\hline
\multirow[c]{7}{*}{BDF2}    &  0.00200000  & $ 1.4 \cdot 10^{-5^{\vphantom{1}}}$  & 36 &  5 &  5 &  5 \\
    &  0.00100000  & $ 3.2 \cdot 10^{-6^{\vphantom{1}}}$  & 38 &  5 &  5 &  5 \\
    &  0.00050000  & $ 9.5 \cdot 10^{-7^{\vphantom{1}}}$  & 19 &  5 &  5 &  5 \\
    &  0.00025000  & $ 1.9 \cdot 10^{-6^{\vphantom{1}}}$  & 18 &  5 &  5 &  5 \\
    &  0.00012500  & $ 1.9 \cdot 10^{-6^{\vphantom{1}}}$  & 18 &  5 &  5 &  5 \\
    &  0.00006250  & $ 1.9 \cdot 10^{-6^{\vphantom{1}}}$  & 18 &  5 &  5 &  5 \\
    &  0.00003125  & $ 1.9 \cdot 10^{-6^{\vphantom{1}}}$  & 18 &  5 &  5 &  5 \\
\end{tabular}

\end{table}

\begin{table}[!htb]
	\footnotesize 
	\caption{\label{tab:ex3b}Computational cost and accuracy obtained when
		solving Example \ref{ex3} for $\alpha = 0.3$ with different numerical
		methods and $\varepsilon = 10^{-8}$.}
	\centering
	\begin{tabular}{l|c|c|c|c|c|c}
  & & & \multicolumn{4}{c}{\bf number of shooting iterations \ \ \ \ \ \ } \\
	    &      &      &  {\bf bisection}  & \multicolumn{3}{c}{\bf proportional secting} \\
	    &      &      & & using             & choosing $\hat c$ in & choosing $\hat c$ in \\
	IVP & step & max. &           & $\hat c  =  1$ & \eqref{eq:nextguess1} according & \eqref{eq:nextguess1} according \\
 solver & size & error &  & in \eqref{eq:nextguess1} & to Remark \ref{rem:propconst}  & to Remark \ref{rem:propconstauto} \\
		\hline
\multirow[c]{7}{*}{Adams}    &  0.00200000  & $ 4.8 \cdot 10^{-6^{\vphantom{1}}}$  & 51 &  6 &  6 &  6 \\
    &  0.00100000  & $ 1.5 \cdot 10^{-6^{\vphantom{1}}}$  & 51 &  6 &  6 &  6 \\
    &  0.00050000  & $ 4.3 \cdot 10^{-7^{\vphantom{1}}}$  & 51 &  6 &  6 &  6 \\
    &  0.00025000  & $ 1.2 \cdot 10^{-7^{\vphantom{1}}}$  & 50 &  6 &  6 &  6 \\
    &  0.00012500  & $ 5.2 \cdot 10^{-8^{\vphantom{1}}}$  & 51 &  6 &  6 &  6 \\
    &  0.00006250  & $ 8.4 \cdot 10^{-9^{\vphantom{1}}}$  & 26 &  6 &  6 &  6 \\
    &  0.00003125  & $ 1.5 \cdot 10^{-8^{\vphantom{1}}}$  & 25 &  6 &  6 &  6 \\
\hline
\multirow[c]{7}{*}{BDF2}    &  0.00200000  & $ 1.3 \cdot 10^{-5^{\vphantom{1}}}$  & 48 &  6 &  6 &  6 \\
    &  0.00100000  & $ 3.2 \cdot 10^{-6^{\vphantom{1}}}$  & 52 &  6 &  6 &  6 \\
    &  0.00050000  & $ 8.3 \cdot 10^{-7^{\vphantom{1}}}$  & 50 &  6 &  6 &  6 \\
    &  0.00025000  & $ 2.0 \cdot 10^{-7^{\vphantom{1}}}$  & 52 &  6 &  6 &  6 \\
    &  0.00012500  & $ 5.2 \cdot 10^{-8^{\vphantom{1}}}$  & 51 &  6 &  6 &  6 \\
    &  0.00006250  & $ 1.3 \cdot 10^{-8^{\vphantom{1}}}$  & 26 &  6 &  6 &  6 \\
    &  0.00003125  & $ 1.5 \cdot 10^{-8^{\vphantom{1}}}$  & 25 &  6 &  6 &  6 \\
\end{tabular}

\end{table}

\begin{table}[!htb]
	\footnotesize
	\caption{\label{tab:ex3c}Computational cost and accuracy obtained when
		solving Example \ref{ex3} for $\alpha = 0.3$ with different numerical
		methods and $\varepsilon = 10^{-10}$.}
	\centering
	\begin{tabular}{l|c|c|c|c|c|c}
  & & & \multicolumn{4}{c}{\bf number of shooting iterations \ \ \ \ \ \ } \\
	    &      &      & {\bf bisection}    & \multicolumn{3}{c}{\bf proportional secting} \\
	    &      &      &              & using                    & choosing $\hat c$ in & choosing $\hat c$ in \\
	IVP & step & max. &           & $\hat c  =  1$ & \eqref{eq:nextguess1} according & \eqref{eq:nextguess1} according \\
 solver & size & error &  & in \eqref{eq:nextguess1} & to Remark \ref{rem:propconst}  & to Remark \ref{rem:propconstauto} \\
		\hline
\multirow[c]{7}{*}{Adams}    &  0.00200000  & $ 4.8 \cdot 10^{-6^{\vphantom{1}}}$  & 63 &  6 &  6 &  6 \\
    &  0.00100000  & $ 1.5 \cdot 10^{-6^{\vphantom{1}}}$  & 64 &  6 &  6 &  6 \\
    &  0.00050000  & $ 4.3 \cdot 10^{-7^{\vphantom{1}}}$  & 65 &  6 &  6 &  6 \\
    &  0.00025000  & $ 1.2 \cdot 10^{-7^{\vphantom{1}}}$  & 65 &  6 &  6 &  6 \\
    &  0.00012500  & $ 3.2 \cdot 10^{-8^{\vphantom{1}}}$  & 63 &  6 &  6 &  6 \\
    &  0.00006250  & $ 8.4 \cdot 10^{-9^{\vphantom{1}}}$  & 64 &  6 &  6 &  6 \\
    &  0.00003125  & $ 2.3 \cdot 10^{-9^{\vphantom{1}}}$  & 65 &  6 &  6 &  6 \\
\hline
\multirow[c]{7}{*}{BDF2}    &  0.00200000  & $ 1.3 \cdot 10^{-5^{\vphantom{1}}}$  & 62 &  6 &  6 &  6 \\
    &  0.00100000  & $ 3.2 \cdot 10^{-6^{\vphantom{1}}}$  & 65 &  6 &  6 &  6 \\
    &  0.00050000  & $ 8.2 \cdot 10^{-7^{\vphantom{1}}}$  & 65 &  6 &  6 &  6 \\
    &  0.00025000  & $ 2.0 \cdot 10^{-7^{\vphantom{1}}}$  & 65 &  6 &  6 &  6 \\
    &  0.00012500  & $ 5.1 \cdot 10^{-8^{\vphantom{1}}}$  & 64 &  6 &  6 &  6 \\
    &  0.00006250  & $ 1.3 \cdot 10^{-8^{\vphantom{1}}}$  & 63 &  6 &  6 &  6 \\
    &  0.00003125  & $ 3.2 \cdot 10^{-9^{\vphantom{1}}}$  & 65 &  6 &  6 &  6 \\
\end{tabular}

\end{table}

The results  in Tables \ref{tab:ex3a}, \ref{tab:ex3b} and \ref{tab:ex3c} show the following:
\begin{itemize}
\item Even if a very small step size is used in the initial value solver, 
	the best accuracy that we can  achieve is limited by the parameter $\varepsilon$ that governs the termination criterion
	of the shooting iterations. Generally the total error is slightly larger than $\varepsilon$ 
	in lockstep with the chosen step size. For relatively small steps the  error component 
	$e_1(k)$ dominates the overall error estimate
	as established in Theorem \ref{thm:err-bound-k} and for larger steps  $e_2(k,N)$ is the 
	dominant error contribution.
\item In Tables \ref{tab:ex3b} and \ref{tab:ex3c} where the accuracy requirement $\varepsilon$ is much smaller than the
	error term $e_2(k,N)$ from Theorem \ref{thm:err-bound-k}, the convergence rate of  BDF2  is around 
	$O(h^2)$ which is exactly the rate  for the corresponding initial value problems, see \cite{Lu1985}.
	This confirms  the expected  behavior predicted by Theorem \ref{thm:err-bound-infty} where we dealt  
	with the case $\varepsilon = 0$.
	For the Adams method, the expected convergence rate is again $O(h^2)$ for  this example, see \cite{Di2003}.
	But the actual rate  in the data is a little lower at $O(h^{1.9})$, because the step size is still too large for  
	asymptotics to have set in.
\item Varying the strategy for selecting the next initial guess (i.e., switching between classical
	bisection and proportional secting) but not changing the IVP solver
	has no influence on the final result  because a change of starting point guessing means  
	trying to solve the same nonlinear equation by a different
	 method which should not lead to  significantly different results.
\item There is, however, a substantial difference in the number of iterations that the two guessing 
	strategies require to obtain a result with the desired accuracy. Classical 
	bisection  needs many iterations, and this induces a rather
	high computational cost. The proportional secting method performs  much better. 
	It typically requires only between 8\% and 24\% of the iterations for  classical bisection.
\item A simple  run times comparison of the two algorithms 
	 reflects this speedup.
	\begin{itemize}
	\item For example, for the case $\varepsilon = 10^{-10}$ shown in Table \ref{tab:ex3c},
		the Adams method (ABM) with  stepsize  0.001 requires 1.53 s to converge when combined with
		classical bisection, but only 0.15~s in combination with the first variant of proportional secting
		with $\hat c = 1$ in eq.~\eqref{eq:nextguess1}, 
		or 0.21 s with either of its two other variants as specified in Remarks \ref{rem:propconst} and
		\ref{rem:propconstauto},	respectively. 
		Here the simplest choice of $\hat c = 1$ is the fastest overall since calculating the quantities $c^*$ 
		and possibly also $c_*$ of eq.~\eqref{eq:propconst} for $\hat c$ is  time consuming. 	
	\item When using the BDF2 solver for the initial value problems with stepsize $0.0005$  we measured run times of 
		3.38 s for the classical bisection method and  0.34 s and 0.43 s, respectively, for the first, 
		and for the second and third  
		variants of the proportional secting algorithm, respectively.
	\end{itemize}
\item The fractional differential equation of Example \ref{ex3}  is not dissipative. For the proportionality factor $\hat c$
	that is required for the second guess for the initial value 
	when using Remark \ref{rem:propconst},  the lower inclusion bound is
	$c_* \approx 0.23$ and the upper bound is $c^* \approx 2.27 \cdot 10^4$. This inclusion interval is
	rather large and  the strategy of Remark~\ref{rem:propconst} 
	gives us a rather inaccurate approximation  of the optimal $\hat c$ for the second guess. 
	Running our new algorithm with the optimal   value for $\hat c$ in the second starting guess 
	converges in a small number of iterations and with 
	the alternative versions of Remark \ref{rem:propconstauto} or $\hat c = 1$ 
	it needs the same small number of iterations.
	The proportional secting method is very forgiving of bad guesses.
\end{itemize}

\begin{example}
	\label{ex5}
	Our second example deals with the terminal value problem
	\[
		D_0^\alpha y(t) =  - \frac 3 2 y(t) \quad \text{with} \quad
		y(7) = \frac {14} 5 E_\alpha \left( - \frac 3 2 \cdot 7^\alpha \right) \approx 0.6476
	\]
	with $\alpha = 0.3$ and  the
	Mittag-Leffler function $E_\alpha$ of order $\alpha$.	The exact solution is 
	\[
		y(t) = \frac {14} 5 E_\alpha \left( - \frac 3 2 t^\alpha \right) .
	\]
\end{example}

The results are in Tables
\ref{tab:ex5a}, \ref{tab:ex5b} and \ref{tab:ex5c}.
For this and the third example below
we  list the data for fewer step sizes than we did for Example \ref{ex3} earlier.

\begin{table}[!htb]
	\footnotesize 
	\caption{\label{tab:ex5a}Computational cost and accuracy obtained when
		solving Example \ref{ex5} for $\alpha = 0.3$ with different numerical
		methods and $\varepsilon = 10^{-6}$.}
	\centering
	\begin{tabular}{l|c|c|c|c|c|c}
  & & & \multicolumn{4}{c}{\bf number of shooting iterations \ \ \ \ \ \ } \\
	    &      &      &           & \multicolumn{3}{c}{\bf proportional secting} \\
	    &      &      & {\bf bisection} &                     & choosing $\hat c$ in & choosing $\hat c$ in \\
	IVP & step & max. &           & using $\hat c = 1$ & \eqref{eq:nextguess1} according & \eqref{eq:nextguess1} according \\
 solver & size & error &  & in \eqref{eq:nextguess1} & to Remark \ref{rem:propconst}  & to Remark \ref{rem:propconstauto} \\
		\hline
\multirow[c]{3}{*}{Adams}    &  0.0140  & $ 5.4 \cdot 10^{-2^{\vphantom{1}}}$  & 21 &  3 &  3 &  3 \\
    &  0.0070  & $ 3.6 \cdot 10^{-2^{\vphantom{1}}}$  & 23 &  3 &  3 &  3 \\
    &  0.0035  & $ 2.5 \cdot 10^{-2^{\vphantom{1}}}$  & 22 &  3 &  3 &  3 \\
\hline
\multirow[c]{3}{*}{BDF2}    &  0.0140  & $ 1.7 \cdot 10^{-5^{\vphantom{1}}}$  & 22 &  3 &  2 &  2 \\
    &  0.0070  & $ 8.1 \cdot 10^{-6^{\vphantom{1}}}$  & 22 &  3 &  2 &  2 \\
    &  0.0035  & $ 1.3 \cdot 10^{-6^{\vphantom{1}}}$  & 23 &  3 &  2 &  2 \\
\end{tabular}

\end{table}

\begin{table}[!htb]
	\footnotesize 
	\caption{\label{tab:ex5b}Computational cost and accuracy obtained when
		solving Example \ref{ex5} for $\alpha = 0.3$ with different numerical
		methods and $\varepsilon = 10^{-8}$.}
	\centering
	\begin{tabular}{l|c|c|c|c|c|c}
  & & & \multicolumn{4}{c}{\bf number of shooting iterations \ \ \ \ \ \ } \\
	    &      &      &           & \multicolumn{3}{c}{\bf proportional secting} \\
	    &      &      & {\bf bisection} &                     & choosing $\hat c$ in & choosing $\hat c$ in \\
	IVP & step & max. &           & using $\hat c = 1$ & \eqref{eq:nextguess1} according & \eqref{eq:nextguess1} according \\
 solver & size & error &  & in \eqref{eq:nextguess1} & to Remark \ref{rem:propconst}  & to Remark \ref{rem:propconstauto} \\
		\hline
\multirow[c]{3}{*}{Adams}    &  0.0140  & $ 5.4 \cdot 10^{-2^{\vphantom{1}}}$  & 29 &  3 &  3 &  3 \\
    &  0.0070  & $ 3.6 \cdot 10^{-2^{\vphantom{1}}}$  & 28 &  3 &  3 &  3 \\
    &  0.0035  & $ 2.5 \cdot 10^{-2^{\vphantom{1}}}$  & 28 &  3 &  3 &  3 \\
\hline
\multirow[c]{3}{*}{BDF2}    &  0.0140  & $ 1.6 \cdot 10^{-5^{\vphantom{1}}}$  & 28 &  3 &  3 &  3 \\
    &  0.0070  & $ 5.1 \cdot 10^{-6^{\vphantom{1}}}$  & 28 &  3 &  3 &  3 \\
    &  0.0035  & $ 1.3 \cdot 10^{-6^{\vphantom{1}}}$  & 23 &  3 &  3 &  3 \\
\end{tabular}

\end{table}

\begin{table}[!htb]
	\footnotesize
	\caption{\label{tab:ex5c}Computational cost and accuracy obtained when
		solving Example \ref{ex5} for $\alpha = 0.3$ with different numerical
		methods and $\varepsilon = 10^{-10}$.}
	\centering
	\begin{tabular}{l|c|c|c|c|c|c}
  & & & \multicolumn{4}{c}{\bf number of shooting iterations \ \ \ \ \ \ } \\
	    &      &      &           & \multicolumn{3}{c}{\bf proportional secting} \\
	    &      &      & {\bf bisection} &                     & choosing $\hat c$ in & choosing $\hat c$ in \\
	IVP & step & max. &           & using $\hat c = 1$ & \eqref{eq:nextguess1} according & \eqref{eq:nextguess1} according \\
 solver & size & error &  & in \eqref{eq:nextguess1} & to Remark \ref{rem:propconst}  & to Remark \ref{rem:propconstauto} \\
		\hline
\multirow[c]{3}{*}{Adams}    &  0.0140  & $ 5.4 \cdot 10^{-2^{\vphantom{1}}}$  & 31 &  3 &  3 &  3 \\
    &  0.0070  & $ 3.6 \cdot 10^{-2^{\vphantom{1}}}$  & 36 &  3 &  3 &  3 \\
    &  0.0035  & $ 2.5 \cdot 10^{-2^{\vphantom{1}}}$  & 36 &  3 &  3 &  3 \\
\hline
\multirow[c]{3}{*}{BDF2}    &  0.0140  & $ 1.6 \cdot 10^{-5^{\vphantom{1}}}$  & 36 &  3 &  3 &  3 \\
    &  0.0070  & $ 5.1 \cdot 10^{-6^{\vphantom{1}}}$  & 36 &  3 &  3 &  3 \\
    &  0.0035  & $ 1.3 \cdot 10^{-6^{\vphantom{1}}}$  & 34 &  3 &  3 &  3 \\
\end{tabular}

\end{table}

Note the following:
\begin{itemize}
\item In Example \ref{ex5}, the exact solution satisfies $y(b) = y(7) \approx 0.6476$ 
	with $y(a) = y(0) = 14/5 = 2.8$, so $y(b)$ is a relatively poor initial guess for $y(0)$.
	All of our shooting start guessing rules perform very similarly to Example  \ref{ex3}.
\item The performance comparisons between bisection and proportional secting 
	are essentially the same as  those for Example \ref{ex3}. 
	Proportional secting 
	reaches the required accuracy with only between 8.3\% and 
	15\% of the iterations that the bisection method needs and the run times
	decrease correspondingly.
\item For both shooting strategies, the accuracy of the BDF2 second-order backward differentiation formula 
	is much better than the accuracy of the ABM method.
\item This example uses a dissipative fractional differential equation that is linear, homogeneous and has the constant 	
	negative growth coefficient $-3/2$.
	For fractional ODEs with constant coefficients the proportionality factor $\hat c$ is discussed
	in Remark \ref{rem:propconst} and it is simple to compute. 
	In this Example  the lower bound $c_*$ coincides with the upper bound $c^*$ and both of them have the
	value $0.23\ldots$ which we can use for $\hat c$. This leads to a slight reduction in the number of required
	iterations.
\end{itemize}

\begin{example}
	\label{ex8}
	The third example is based on the initial value problem
	\[
		D_0^\alpha y(t) = \frac 1 {t+1} \sin (t \cdot y(t)), \qquad
		y(0) = 1,
	\]
	with $\alpha = 0.7$ on the interval $[a,b] = [0, 20]$. 
\end{example}

The exact solution for this problem is unknown. Using  a second order backward differentiation formula with 16,000,000 
steps on $[a,b] = [0, 20]$, and thus a stepsize of $ 20/(16 \cdot 10^6)  = 1.25 \cdot 10^{-6}$, 
we  have computed the approximate
solution shown in Figure \ref{fig:alg} as the curve highlighted by the black dots. 
We are confident to be very 
close to the exact solution with terminal value $y(b) = y(20) \approx 0.8360565$.
By replacing the given initial condition above by the terminal condition, we obtain a terminal
value problem that we  try to solve with the fractional ODE shooting methods  of this paper. 
This fractional ODE problem appears  more challenging than
those of Examples \ref{ex3} and \ref{ex5} because here we work on a much 
larger interval and with a decaying oscillatory exact solution.
The results are listed in Tables
\ref{tab:ex8a}, \ref{tab:ex8b} and \ref{tab:ex8c}.
Without any precise information of the exact solution, the listed 
computed errors are in comparison to the numerical solution constructed earlier.

\begin{table}[!htb]
	\footnotesize
	\caption{\label{tab:ex8a}Computational cost and accuracy obtained when
		solving Example \ref{ex8} for $\alpha = 0.7$ with different numerical
		methods and $\varepsilon = 10^{-6}$.}
	\centering
	\begin{tabular}{l|c|c|c|c|c|c}
  & & & \multicolumn{4}{c}{\bf number of shooting iterations \ \ \ \ \ \ } \\
	    &      &      &           & \multicolumn{3}{c}{\bf proportional secting} \\
	    &      &      & {\bf bisection} &                     & choosing $\hat c$ in & choosing $\hat c$ in \\
	IVP & step & max. &           & using $\hat c = 1$ & \eqref{eq:nextguess1} according & \eqref{eq:nextguess1} according \\
 solver & size & error &  & in \eqref{eq:nextguess1} & to Remark \ref{rem:propconst}  & to Remark \ref{rem:propconstauto} \\
		\hline
\multirow[c]{3}{*}{Adams}    &  0.0400  & $ 2.0 \cdot 10^{-4^{\vphantom{1}}}$  & 16 &  7 &  8 &  7 \\
    &  0.0200  & $ 5.0 \cdot 10^{-5^{\vphantom{1}}}$  & 20 &  7 &  8 &  7 \\
    &  0.0100  & $ 1.4 \cdot 10^{-5^{\vphantom{1}}}$  & 18 &  7 &  8 &  7 \\
\hline
\multirow[c]{3}{*}{BDF2}    &  0.0400  & $ 5.1 \cdot 10^{-4^{\vphantom{1}}}$  & 16 &  7 &  8 &  7 \\
    &  0.0200  & $ 1.3 \cdot 10^{-4^{\vphantom{1}}}$  & 19 &  7 &  8 &  7 \\
    &  0.0100  & $ 3.4 \cdot 10^{-5^{\vphantom{1}}}$  & 20 &  7 &  8 &  7 \\
\end{tabular}

\end{table}

\begin{table}[!htb]
	\footnotesize
	\caption{\label{tab:ex8b}Computational cost and accuracy obtained when
		solving Example \ref{ex8} for $\alpha = 0.7$ with different numerical
		methods and $\varepsilon = 10^{-8}$.}
	\centering
	\begin{tabular}{l|c|c|c|c|c|c}
  & & & \multicolumn{4}{c}{\bf number of shooting iterations \ \ \ \ \ \ } \\
	    &      &      &           & \multicolumn{3}{c}{\bf proportional secting} \\
	    &      &      & {\bf bisection} &                     & choosing $\hat c$ in & choosing $\hat c$ in \\
	IVP & step & max. &           & using $\hat c = 1$ & \eqref{eq:nextguess1} according & \eqref{eq:nextguess1} according \\
 solver & size & error &  & in \eqref{eq:nextguess1} & to Remark \ref{rem:propconst}  & to Remark \ref{rem:propconstauto} \\
		\hline
\multirow[c]{3}{*}{Adams}    &  0.0400  & $ 2.0 \cdot 10^{-4^{\vphantom{1}}}$  & 25 &  7 &  8 &  7 \\
    &  0.0200  & $ 5.0 \cdot 10^{-5^{\vphantom{1}}}$  & 26 &  7 &  8 &  7 \\
    &  0.0100  & $ 1.2 \cdot 10^{-5^{\vphantom{1}}}$  & 24 &  7 &  8 &  7 \\
\hline
\multirow[c]{3}{*}{BDF2}    &  0.0400  & $ 5.1 \cdot 10^{-4^{\vphantom{1}}}$  & 26 &  7 &  8 &  7 \\
    &  0.0200  & $ 1.3 \cdot 10^{-4^{\vphantom{1}}}$  & 27 &  7 &  8 &  7 \\
    &  0.0100  & $ 3.4 \cdot 10^{-5^{\vphantom{1}}}$  & 27 &  7 &  8 &  7 \\
\end{tabular}

\end{table}

\begin{table}[!htb]
	\footnotesize
	\caption{\label{tab:ex8c}Computational cost and accuracy obtained when
		solving Example \ref{ex8} for $\alpha = 0.7$ with different numerical
		methods and $\varepsilon = 10^{-10}$.}
	\centering
	\begin{tabular}{l|c|c|c|c|c|c}
  & & & \multicolumn{4}{c}{\bf number of shooting iterations \ \ \ \ \ \ } \\
	    &      &      &           & \multicolumn{3}{c}{\bf proportional secting} \\
	    &      &      & {\bf bisection} &                     & choosing $\hat c$ in & choosing $\hat c$ in \\
	IVP & step & max. &           & using $\hat c = 1$ & \eqref{eq:nextguess1} according & \eqref{eq:nextguess1} according \\
 solver & size & error &  & in \eqref{eq:nextguess1} & to Remark \ref{rem:propconst}  & to Remark \ref{rem:propconstauto} \\
		\hline
\multirow[c]{3}{*}{Adams}    &  0.0400  & $ 2.0 \cdot 10^{-4^{\vphantom{1}}}$  & 32 &  8 &  9 &  8 \\
    &  0.0200  & $ 5.0 \cdot 10^{-5^{\vphantom{1}}}$  & 33 &  8 &  9 &  8 \\
    &  0.0100  & $ 1.2 \cdot 10^{-5^{\vphantom{1}}}$  & 33 &  8 &  9 &  8 \\
\hline
\multirow[c]{3}{*}{BDF2}    &  0.0400  & $ 5.1 \cdot 10^{-4^{\vphantom{1}}}$  & 30 &  8 &  9 &  8 \\
    &  0.0200  & $ 1.3 \cdot 10^{-4^{\vphantom{1}}}$  & 33 &  8 &  9 &  8 \\
    &  0.0100  & $ 3.4 \cdot 10^{-5^{\vphantom{1}}}$  & 33 &  8 &  9 &  8 \\
\end{tabular}

\end{table}

These tables again exhibit a similar behavior as our  Examples \ref{ex3} and \ref{ex5}. The
proportional secting method is substantially faster than the classical bisection method. 
It requires significantly fewer shooting iterations to converge for the required accuracy. 
The same holds true for the run time measurements. For  $\varepsilon = 10^{-8}$ 
(see Table \ref{tab:ex8b}) and a BDF2 solver with stepsize $0.02$, the run time is 
0.48~s for the classical bisection method while  proportional secting  
needs only 0.14~s for the variant with $\hat c = 1$ in eq.~\eqref{eq:nextguess1} 
and 0.18~s when $\hat c$ in eq.~\eqref{eq:nextguess1} is chosen according to  Remark \ref{rem:propconst}
and 0.15~s when the idea of Remark \ref{rem:propconstauto} is used to compute $\hat c$.

There is a significant difference in Example \ref{ex8}  compared to Example \ref{ex5}:
Using the strategy of Remark \ref{rem:propconst} to compute the 
second guess for the initial value leads to convergence,
but it is slightly worse than when simply using $\hat c = 1$
as suggested in Remark \ref{rem:hatc-simple}.
This fractional differential equation is not dissipative, and its
containment interval bounds given by \eqref{eq:propconst} are
$c_* \approx 0.05$ and $c^* \approx 5 \cdot 10^7$. Thus, the first containing interval 
for the correct proportionality factor is extremely large
and the midpoint method of Remark \ref{rem:propconst} then starts from a very
large error and continues with a relatively poor second approximate solution so that more iterations
are needed until convergence.

\section{Conclusion}

We have discussed shooting methods for solving fractional 
terminal value problems  with Caputo derivatives numerically.  Choosing the best numerical
IVP solver was not our focus. This has already been discussed extensively 
in \cite{FM2011,Ga2010,Ga2015,Ga2018}.
Instead we have investigated and tested  
algorithms that select initial values for each iteration in fractional ODE shooting
procedures. Classical bisection is often recommended and most often used. It converges rather slowly, 
requiring  many iterations until approximating  the actual solution well. 
The newly proposed proportional secting method is much better here.
It computes the guess of the second and subsequent starting values 
for the shooting procedure rather differently. Three separate 
methods, differing only in the  
guess of the second initial value, have been proposed 
and their respective performances shown to differ only slightly in speed and accuracy.
Our experimental findings are supported by the results of analytical investigations.

\begin{remark}
	The Caputo differential operator listed in eq.~\eqref{eq:def-caputo} is not the only 
	fractional differential operator of order $\alpha$ with starting point $a$ 
	for which one can try to formulate
	terminal value problems. Indeed, it is conceivable to use the so-called Hilfer
	fractional derivatives of order $\alpha$ and type $\mu \in [0,1]$ with starting point $a$ 
	\cite[Definition 3.3]{Hilfer} instead. The special case $\mu = 1$ of this class of operators is, on a very large set 
	of functions, equivalent to the Caputo operators that we have discussed here;
	for $\mu = 0$ one obtains the Riemann-Liouville derivatives \cite[Chapter 2]{Di2010}.
	In principle, one could use the approach that we have proposed here to handle such
	generalized terminal value problems. One would then have to replace the initial value
	problem solving subroutines for Caputo IVPs in Algorithm \ref{alg:propsec}
	by corresponding functions for the modified operators, which is a relatively straightforward matter. 
	However, to theoretically justify the proportional secting idea in this generalized
	context, one would also need to generalize our Corollary~\ref{cor:bounds}
	in a corresponding way. Such a result currently does not seem to be available.
\end{remark}

\section*{Software}
The MATLAB source codes of the algorithms described in this paper, including all required
auxiliary functions, can be downloaded from a dedicated repository \cite{DGU2023} on the Zenodo 
platform, thus allowing all  readers to reproduce the results of our numerical experiments.
Our functions were tested in MATLAB R2022a and  MATLAB R2022b.

\section*{Declarations}
{
\small

\noindent{\bf Funding}
The authors did not receive support from any organization for the submitted work.
\medskip

\noindent{\bf Declaration of Interests}
The authors declare that they have no known competing financial interests 
or personal relationships that could have appeared to influence the work reported in this paper.
\medskip

\noindent{\bf Author Contributions}
Both authors contributed equally to this work.
\medskip

\noindent{\bf Data Availability}
Not applicable.
}

\begin{acknowledgements}
The work described in this paper originated from a discussion between the two authors that
was initiated after a Zoom talk by the first author  in the Irish Numerical Analysis Forum in 2021. 
We are grateful to the organizers Natalia Kopteva and Martin Stynes for their Numerical Analysis Forum 
that brought us together by chance.

Our algorithm uses some functions developed and implemented by Roberto Garrappa
as auxiliary subroutines, e.g.\ for the
evaluation of the Mittag-Leffler function \cite{Ga2015b} and for solving fractional initial value 
problems \cite{Ga2018}. These
routines are also of interest in their own right. We thank Roberto Garrappa for allowing us to include
his codes into our software suite~\cite{DGU2023}.
\end{acknowledgements}

\bibliographystyle{spmpsci}
\bibliography{Diethelm-Uhlig}

\end{document}